\documentclass[12pt]{amsart}
\usepackage[utf8]{inputenc}
\usepackage{bbm}
\usepackage[english]{babel}
\usepackage{amsthm, amsfonts, amssymb, amscd, fleqn}
\usepackage{nicefrac}
\usepackage{tikz}
\usepackage{enumerate}

\usetikzlibrary{arrows}
\usetikzlibrary{calc}
\usepackage{graphicx}
\usepackage[hyphens]{url}
\usepackage{babelbib}
\usepackage{etoolbox}

\newtheorem{theorem}{Theorem}[section]

\newtheorem{corollary}[theorem]{Corollary}
\newtheorem{lemma}[theorem]{Lemma}

\theoremstyle{definition}
\newtheorem{remark}[theorem]{Remark}

\newtheorem{fact}[theorem]{Fact}

\theoremstyle{definition}

\theoremstyle{remark}

\newcommand\restr[2]{{
  \left.\kern-\nulldelimiterspace 
  #1 
  \vphantom{\big|} 
  \right|_{#2} 
  }}
\newcommand{\vertiii}[1]{{\left\vert\kern-0.25ex\left\vert\kern-0.25ex\left\vert #1 
    \right\vert\kern-0.25ex\right\vert\kern-0.25ex\right\vert}}

\newcommand\norm[1]{\|#1\|}
\newcommand\vn[1]{|#1|^{\frac{1}{n}}}

\def\vr{\operatorname{vr}}

\def\absconv{\operatorname{absconv}}

\def\x{\mathbb{x}}
\def\+{\mathbb{+}}

\def\R{\mathbb{R}}
\def\E{\mathbb{E}}

\def\N{\mathbb{N}}

\def\P{\mathbb{P}}
\def\PP{\mathcal{P}}

\def\NN{\mathcal{N}}

\def\PP{\mathcal{P}}

\def\SS{\mathcal{S}}

\def\LL{\mathcal{L}}
\def\HH{\mathcal{H}}

\newcommand{\Pro}{\mathbb P}

\newcommand{\eps}{\varepsilon}

\def\r{\right}

\makeatletter
\@namedef{subjclassname@2020}{%
  \textup{2020} Mathematics Subject Classification}
\makeatother

\title[ On the volume ratio of projections of convex bodies]{On the volume ratio of projections of convex bodies}

\author[D. Galicer]{Daniel Galicer}
\address{ Departamento de Matem\'{a}tica - IMAS-CONICET,
Facultad de Cs. Exactas y Naturales  Pab. I, Universidad de Buenos Aires
(1428) Buenos Aires, Argentina}
\email{dgalicer@dm.uba.ar}

\author[A. Litvak]{Alexander E. Litvak}
\address{ Dept.~of Math.~and Stat.~Sciences,
University of Alberta, Edmonton, AB, Canada, T6G 2G1.}
\email{alitvak@ualberta.ca}

\author[M. Merzbacher]{Mariano Merzbacher}
\address{ Departamento de Matem\'{a}tica - IMAS-CONICET,
Facultad de Cs. Exactas y Naturales  Pab. I, Universidad de Buenos Aires
(1428) Buenos Aires, Argentina}
\email{mmerzbacher@dm.uba.ar}

\author[D. Pinasco]{Dami\'an Pinasco}
\address{Departamento de Matem\'{a}ticas y Estad\'{\i}stica, Universidad T. Di Tella, Av. Figueroa Alcorta 7350 (1428), Buenos Aires, Argentina and CONICET}
\email{dpinasco@utdt.edu}
\begin{document}



\begin{abstract}
We study the volume ratio between \emph{projections} of two convex bodies.
Given a high-dimensional convex body $K$ we show that there is another convex body
$L$ such that the volume ratio between any two projections of fixed rank of the bodies
$K$ and $L$ is large. Namely, we prove that for every $1\leq k\leq n$ and for
each convex body $K\subset \R^n$
there is a centrally symmetric body $L \subset \R^n$   such that for any two
 projections $P, Q: \R^n \to \R^n$ of rank $k$ one has
$$
  \vr(PK, QL) \geq  c \, \min\left\{\frac{  k}{  \sqrt{n}} \, \sqrt{\frac{1}{\log \log \log(\frac{n\log(n)}{k})}}, \,
  \frac{\sqrt{k}}{\sqrt{\log(\frac{n\log(n)}{k})}}\r\},
$$
 where $c>0$ is an absolute constant.
This general lower bound is sharp (up to logarithmic factors)
in the regime $k\geq n^{2/3}$.

\end{abstract}

\maketitle
\section{Introduction.}
The problem of estimating Banach-Mazur distances  between projections or sections of convex bodies had aroused considerable interest
(see  for example \cite{BM-quot, mankiewicz2001geometry, LMT, rudelson2004extremal} and references therein).
Recall that this distance, for two centrally symmetric convex bodies $K$ and $L$ in $\mathbb{R}^n$,  is defined as
 \begin{align}
 d_{BM}(K,L) = \inf\{a \cdot b  \, \, \, \, | \, \,\,  \,  \frac{1}{a}K \subset TL \subset bK\},
 \end{align}
 where the infimum is taken over all invertible linear operators $T:\R^n \to \R^n$ and all
 $a, b>0$.

Note that  two convex bodies can be far apart but they may have their projections or sections quite close.
This happens, for example, if the bodies are  Gluskin's polytopes (absolute convex hulls of, say,
$3n$ random points on the standard Euclidean sphere in $\R^n$). Indeed, Gluskin \cite{Gl-81} proved that with high probability
the Banach--Mazur distance between two such polytopes is at least $c n$, where $c$ is an absolute positive constant. On the other hand
it is known that ``most'' sections of a Gluskin polytope are nearly Euclidean, thus ``most'' sections of two Gluskin's polytopes
are quite close to each other. This follows from results on sections of convex bodies having bounded volume ratio
\cite{szarek1980nearly, Sz-vr},
see below for the precise definitions. A more general question was studied by Mankiewicz and Tomczak-Jaegermann in \cite{mankiewicz2001geometry}.

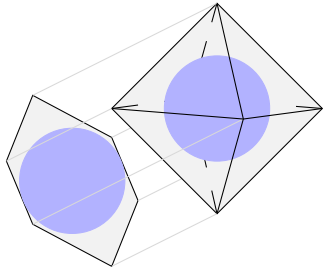
\begin{figure}[ht]
\begin{center}
\begin{tikzpicture}[scale=0.7]
\filldraw[fill=black!5!white](-2,-3)--(-3.5,-2.2)--(-4,-1)--(-3.5,0.25)--(-2,
-0.55)--(-1.5,-1.75)--(-2,-3);
\draw[black!15!white](2,0)--(-1.5,-1.75);
\draw[black!15!white](-2,-0.55)--(-0.5,0.2);
\filldraw[fill=black!5!white](2,0)--(0,-2)--(-2,0)--(0,2)--(2,0);
\draw (-2,0)--(-1.5,0.066);
\draw(-1,0.133)--(-0.5,0.2);
\draw(-0.5,0.2)--(0,0.16);
\draw(0.5,0.12)--(1,0.08);
\draw(1.5,0.04)--(2,0);
\draw(-0.5,0.2)--(-0.4,0.56);
\draw(-0.3,0.92)--(-0.2,1.28);
\draw(-0.1,1.64)--(0,2);
\draw(-0.5,0.2)--(-0.4,-0.24);
\draw(-0.3,-0.68)--(-0.2,-1.12);
\draw(-0.1,-1.56)--(0,-2);
\filldraw[blue!30!white] (0,0) circle (1 cm);
\filldraw[blue!30!white] (-2.75,-1.375) circle (1 cm);
\draw[black!15!white](0,2)--(-3.5,0.25);
\draw[black!15!white](0,-2)--(-2,-3);
\draw[black!15!white](-2,0)--(-4,-1);
\draw[black!15!white](0.5,-0.2)--(-3.5,-2.2);
\draw (-2,0)--(0.5,-0.2)--(2,0);
\draw (0,2)--(0.5,-0.2)--(0,-2);
\end{tikzpicture}
\caption{A projection of the octahedron and the Euclidean ball. }
\end{center}
\end{figure}

They estimated average distance between random $k$-dimensional projections of two given centrally symmetric convex bodies.
It turns out that such an average is bounded below by the product of averages of distances of $((\frac{1}{2}-\eps) k)$-dimensional
projections of these bodies to the Euclidean ball.
Note here that a Gluskin's polytope can be viewed as a random projection of, say, $(3n)$-dimensional octahedron to $\R^n$.

Rudelson \cite{rudelson2004extremal} studied the problem of estimating extremal distances between projections
of centrally symmetric convex bodies. For $k < n$ define the distance $\delta_k(K, L)$ as the minimal Banach--Mazur
distance between $k$-dimensional projections of $K$ and $L$. Rudelson was interested in estimating the diameter of
the Banach--Mazur compactum for this distance; that is, in finding the asymptotic behaviour of
\begin{align*}
\Delta(k,n):= \sup \delta_k(K,L),
\end{align*}
where the supremum is taken over all $n$-dimensional convex symmetric bodies $K$ and $L$.
   He proved that
\begin{align}\label{Rudbou}
\Delta(k,n)\sim_{\log n} \begin{cases}
\sqrt{k} & \text{if } k\leq n^{\nicefrac{2}{3}}\\
\frac{k^2}{n} &\text{if } k > n^{\nicefrac{2}{3}},
\end{cases}
\end{align}
where $A \sim_{\log n} B$ means that
$$
  \frac{1}{C\log^a n} A \leq B \leq (C \log ^a n) A
$$
for some absolute constants $C,a > 0$. In particular, Rudelson showed that
 there are two centrally symmetric convex bodies $K,L \subset \R^n$, such that for any $k < n$,
$$
 \delta_k(K, L) \gtrsim \frac{k^2}{n\log \log n},
$$
where $A \gtrsim B$ means that $A\geq c B$ for some absolute constant $c > 0$.
Note also a well-known fact (proved in  \cite{BLM, CP, gluskin1989extremal, gluskin1989deviation})
$$
  \delta_k( B_2^n, B_1^n) \gtrsim  \sqrt{\frac{k}{\log(1 + \frac{n}{k})}}.
$$


Another possible measure of how far two convex bodies $K, L\subset \R^n$ are from each other, is given by their \emph{volume ratio}:
\begin{align*}
 \vr(K,L):= \inf \left\{ \left(\frac{|K|}{|T(L)|}\right)^{\frac{1}{n}}\,\, |\,\,
  T : \R^n \to \R^n \,\, \mbox{ is an affine transformation, } \,\,
  T(L)\subset  K  \right\},
 \end{align*}
  where $|\cdot|$ denotes $n$-dimensional volume.
Note that the standard volume ratio $\vr(K)$ introduced in \cite{szarek1980nearly}
is just $\vr(K,B_2^n)$.

In other words, $\vr(K,L)$ measures how well $K$ can be approximated from inside by an affine image of $L$ in terms of volume.
This invariant  goes back to the works of McBeath \cite{macbeath1951compactness} and Levi~\cite{flevi}. It was further
investigated by many authors. In particular, Giannopoulos and Hartzoulaki \cite{giannopoulos2002volume} proved that
for every two convex bodies $K, L \subset \R^n$,
\begin{align}\label{GH-vr-orig}
 \vr(K,L)\leq C \sqrt{n} \, \log n,
\end{align}
where $C>0$ is an absolute constant. On the other hand,
it was proved in \cite{Gamepi2021} that given a convex body $K \subset \mathbb{R}^n$ there is centrally symmetric body $L \subset \mathbb{R}^n$ such that the volume ratio $\vr(K,L)$ is large. Precisely, we have
\begin{align} \label{lowerbound}
 \vr(K,L)\geq C \sqrt{n},
\end{align}
where $C>0$ is an absolute constant.
 This general lower estimate is sharp: by John's theorem and a reduction to the symmetric case we have, for example, that given any convex body $L\subset \mathbb{R}^n$, $\vr(B_{2}^n,L) \leq \sqrt{n}$.  The lower bound in \eqref{lowerbound} is a refinement of a previous estimate  obtained by Khrabrov in \cite{khrabrov2001generalized} of order $\sqrt{\frac{n}{\log \log (n)}}$.

 We would also like to note that
$$
  d_m(K,L) = \vr(K,L)\vr(L, K)
$$
is a weaker version of the Banach--Mazur distance, called {\it modified Banach--Mazur distance}
(this name comes from \cite{Khrabrov2000}).
Clearly, $d_m(K,L)\leq  d(K,L)$.
It was introduced in \cite{macbeath1951compactness}
(in fact, logarithm of it, see also \cite{flevi}) and then implicitly
used in \cite{GKM, Gl-81} in order to estimate the Banach--Mazur distance from below. Then it was
investigated in series of works by Khrabrov, see also Corollary~5.3 and Remark~5.4 in \cite{gordon2004john}.
Moreover, Khrabrov \cite{khrabrov2001generalized} proved that for every
centrally symmetric convex body $K \subset \R^n$ and every $1\leq p\leq \infty$,
\begin{align}\label{GH-vr}
  d_m(K, B_p^n) = \vr(K, B_p^n)\vr(B_p^n, K)\leq \sqrt{e n}.
\end{align}



We extend the notion of volume ratio for two given bodies lying in different subspaces of $\R^n$
in the following natural way. Let $1\leq k\leq n$ and let $E, F$ be two $k$-dimensional subspaces of $\R^n$.
Then for two convex bodies  $K\subset E$ and $L\subset F$, we define
\begin{align*}
 \vr(K,L):= \inf \left\{ \left(\frac{|K|}{|T(L)|}\right)^{\frac{1}{k}}\,\, |\,\, T : E \to F \,\, \mbox{ is an affine transformation, } \,\, T(L) \subset K  \right\},
 \end{align*}
 where $|\cdot|$ denotes $k$-dimensional volume.

 Note that for a convex body $K$ we have a collection of $k$-dimensional convex bodies given by  $QK \subset \R^k$ for any given projection  $Q: \R^n \to \R^n$ of rank $k$.
Here we provide a lower bound for volume ratio in the spirit of Rudelson's approach.
 Namely, we show that for every high-dimensional convex body $K \subset \R^n$
there exists a centrally symmetric convex body $L\subset \R^n$ such that, for every pair of $k$-dimensional projections $P, Q: \R^n \to \R^n$, the volume ratio $\vr(PK, QL)$ is large.
The following theorem is the main result of this work.

\begin{theorem}\label{main thm}
Let $n$ be large enough and $k \leq n$. Then
for every convex body $K\subset \R^n$ there is a centrally symmetric body $L \subset \R^n$   such that
for any two $k$-dimensional projections $P, Q: \R^n \to \R^n$  one has
$$
 \vr(P K, Q L) \geq c
    \, \min\left\{\frac{  k}{  \sqrt{n}} \cdot \sqrt{\frac{1}{\log \log \log(\frac{n\log(n)}{k})}}, \,
  \frac{\sqrt{k}}{\sqrt{\log(\frac{n\log(n)}{k})}}\right\}  ,
$$
where $c>0$ is an absolute constant.
\end{theorem}

Moreover, in Corollary~\ref{sharp} below, we show that Theorem~\ref{main thm} is sharp
(up to logarithmic factors) in the regime $k\geq n^{2/3}$. Remarkably, the phase transition in \eqref{Rudbou} is exactly  $k\sim n^{2/3}$.

Although it is not directly related, we would like to mention the following result from
\cite{GLT}. Interestingly, it also uses Gluskin's polytopes in the proof and leads to
the  essentially same lower bound.
 For all $1\leq k\leq n$ there exist a convex body  $K \subset \R^n$
such that for every centrally symmetric convex body  $L \subset \R^n$ and any
$k$-dimensional projections $P$ and $Q$ one has
$$
  d_{BM}(P K, Q L) \geq   \frac{c\, k}{\sqrt{n  \ln n }},
$$
where $c>0$ is an absolute constant.


Finally, we mention that using duality we can reformulate our theorem in terms of sections.
For a convex body $K$ with the origin in its interior and a subspace $E \subset \R^n$ one has
 $P_E (K^\circ) = (E \cap K)^\circ$, where $P_E$ denotes the orthogonal projection onto $E$.
 Note that for centrally symmetric bodies it is enough to consider only linear operators in
 the definition of volume ratio. Therefore, every result concerning volume ratios of projections
 of a centrally symmetric bodies  $K$ and $L$ has a dual version concerning volume ratios of
 sections of $K^\circ$ and $L^\circ$. Thus, we can also state a dual version of
 our previous result.

\begin{corollary}\label{dual thm}
Let $1\leq k\leq n$.
For each centrally symmetric convex body
$K\subset \R^n$
there is a centrally symmetric body $L \subset \R^n$   such that
 for any two $k$-dimensional subspaces $E, F \subset \R^n$  one has
$$
 \vr(F \cap L, E\cap K) \geq c\,  \min\left\{\frac{  k}{  \sqrt{n}} \cdot \sqrt{\frac{1}{\log \log
 \log(\frac{n\log(n)}{k})}}, \,
  \frac{\sqrt{k}}{\sqrt{\log(\frac{n\log(n)}{k})}}\right\},
$$
  where $c>0$ is an absolute constant.
\end{corollary}

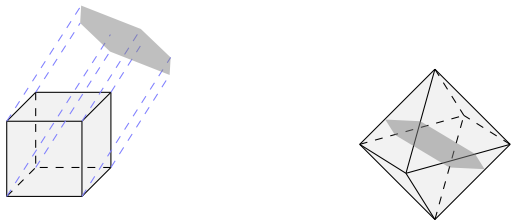
\begin{figure}[h]
\begin{center}
\begin{tikzpicture}\begin{scope}[xshift = 5cm,line join=round]
\coordinate (A) at (1,0,0);
\coordinate (B) at (0,1,0);
\coordinate (C) at (0,0,1);
\coordinate (D) at (-1,0,0);
\coordinate (E) at (0,-1,0);
\coordinate (F) at (0,0,-1);

\fill [white!90!gray] (A) -- (B) -- (C) --(A);
\fill [white!90!gray] (B) -- (D) -- (C) --(B);
\fill [white!90!gray] (A) -- (C) -- (E) --(A);
\fill [white!90!gray] (D) -- (C) -- (E) --(D);
\draw (A) -- (B);
\draw (A) -- (C);
\draw [dashed](A) -- (F);
\draw (A) -- (E);
\draw (D) -- (B);
\draw (D) -- (C);
\draw[dashed] (D) -- (F);
\draw (D) -- (E);
\draw (B) -- (C);
\draw [dashed](B) -- (F);
\draw (E) -- (C);
\draw [dashed](E) -- (F);



\coordinate (U1) at (0.66,-0.33,0);
\coordinate (U2) at (-0.66,0.33,0);
\coordinate (U3) at (0.66,0,0.33);
\coordinate (U4) at (-0.66,0,-0.33);
\coordinate (U5) at (0,0.5,0.5);
\coordinate (U6) at (0,-0.5,-0.5);


\fill[color=gray,opacity = 0.5] (U4) -- (U2) -- (U5) -- (U3) --(U1)-- (U6) -- (U4);

\end{scope}
\begin{scope}[ scale=0.5]
\coordinate (Ac) at (1,1,1);
\coordinate (Bc) at (1,1,-1);
\coordinate (Cc) at (1,-1,1);
\coordinate (Dc) at (1,-1,-1);
\coordinate (Ec) at (-1,1,1);
\coordinate (Fc) at (-1,1,-1);
\coordinate (Gc) at (-1,-1,1);
\coordinate (Hc) at (-1,-1,-1);
\coordinate (A) at (1,-1,0);
\coordinate (B) at (0,-1,0);
\coordinate (C) at (0,-1,1);
\coordinate (D) at (-1,-1,0);
\coordinate (E) at (0,-1,0);
\coordinate (F) at (0,-1,-1);
\fill[white!90!gray] (Ac) -- (Bc) -- (Dc) -- (Cc) -- (Ac); 
\fill[white!90!gray] (Ac) -- (Ec) -- (Fc) -- (Bc) -- (Ac);
\fill[white!90!gray] (Ac) -- (Cc) -- (Gc) -- (Ec) -- (Ac);  
\draw (Ac) -- (Bc);
\draw (Bc) -- (Dc);
\draw (Ac) -- (Cc);
\draw (Cc) -- (Dc);
\draw (Ec) -- (Fc);
\draw (Ec) --(Ac);
\draw (Fc) -- (Bc);
\draw (Gc) -- (Cc);
\draw (Ec) -- (Gc);
\draw [dashed] (Fc) -- (Hc)--(Dc);
\draw [dashed](Gc) -- (Hc);

\end{scope}

\begin{scope}[scale = 0.5]
\def\x{1};
\def\y{2};
\def\z{-2};
\coordinate (pAc) at ({0.88 + \x},{0.77 + \y},{1.22 + \z});
\coordinate (pBc) at ({0.44 + \x},{-0.11 + \y},{0.11 + \z});
\coordinate (pCc) at ({1.33 + \x},{-0.33 + \y},{0.33 + \z});
\coordinate (pDc) at ({0.88 + \x},{-1.22 + \y},{-0.77 + \z});
\coordinate (pEc) at ({-0.88 + \x},{1.22 + \y},{0.77 + \z});
\coordinate (pFc) at ({-1.33 + \x},{0.33 + \y},{-0.33 + + \z});
\coordinate (pGc) at ({-0.44+ \x},{0.11 + \y},{-0.11 + \z});
\coordinate (pHc) at ({-0.88+ + \x},{-0.77 + \y},{-1.22 + \z});

\coordinate (Ac) at (1,1,1);
\coordinate (Bc) at (1,1,-1);
\coordinate (Cc) at (1,-1,1);
\coordinate (Dc) at (1,-1,-1);
\coordinate (Ec) at (-1,1,1);
\coordinate (Fc) at (-1,1,-1);
\coordinate (Gc) at (-1,-1,1);
\coordinate (Hc) at (-1,-1,-1);

\draw[dashed,blue!50!white] (Ac) -- (pAc);
\draw[dashed,blue!50!white] (Bc) -- (pBc);
\draw[dashed,blue!50!white] (Cc) -- (pCc);
\draw[dashed,blue!50!white] (Dc) -- (pDc);
\draw[dashed,blue!50!white] (Ec) -- (pEc);
\draw[dashed,blue!50!white] (Fc) -- (pFc);
\draw[dashed,blue!50!white] (Gc) -- (pGc);
\draw[dashed,blue!50!white] (Hc) -- (pHc);

%

\fill[color=gray,opacity = 0.5]  (pAc) -- (pCc)--(pDc)--(pHc)-- (pFc) -- (pEc);

\end{scope}

\end{tikzpicture}
\caption{A projection of $B_\infty^n$ and a section of $B_1^n=(B_\infty^n)^{\circ}$. }
\end{center}
\end{figure}

\section{Preliminaries.}\label{Spreliminaries}

Given two sequences of real numbers $(a_{n})_{n}$ and $(b_{n})_{n}$ we write $a_{n} \lesssim b_{n}$
(resp., $a_{n} \gtrsim b_{n}$) if there  exists an absolute constant $C>0$ (independent of $n$)
such that $a_{n} \leq C b_{n}$ (resp., $C a_{n} \geq b_{n}$) for every $n$.
We write $a_{n} \sim b_{n}$ if $a_{n} \lesssim b_{n}$ and $b_{n} \lesssim a_{n}$.
We denote by $e_1, \dots, e_n$ the canonical vector basis in $\R^n$ and by $B_2^n$ and $S^{n-1}$,
the unit ball and unit sphere in $\R^n$. Similarly, the unit ball of $\ell_p^n$ is denoted by $B_p^n$,
where the norm in  $\ell_p^n$ is defined by
$$
 \|x\|_p=\left(\sum_{i=1}^n |x_i|^p\right)^{1/p}\quad \mbox{for}\,\, 1\leq p< \infty \quad
 \mbox{and} \quad \|x\|_\infty = \max _{i\leq n} |x_i|.
$$
Given $X_1, \dots, X_m\in \R^n$, we denote by $\absconv \{X_1, \dots, X_m \}$ their absolute convex hull, that is,
$$\absconv\{X_1, \dots, X_m\} : = \left\{ \sum_{i=1}^m a_i X_i\,\, \, \,  |\, \,\,\,  \sum_{i=1}^m \vert a_i \vert \leq 1 \right\} \subset \R^n.$$

A convex body $K \subset \R^n$ is a compact convex set with non-empty interior.
Its Minkowski functional is defined on
$\R^n$ by
 $$
    \|x\|_K  = \inf \left\{\lambda >0 \,\, |\, \, x \in \lambda K \right\}.
 $$
If $K$ is centrally symmetric (i.e., $K = -K$), then $\|\cdot\|_K$ defines a norm and
we denote by $X_K$ the normed space $(\R^n, \Vert \cdot \Vert_{K})$ that has $K$ as its unit ball.
By $|K|$ we denote the $n$-dimensional volume of $K$. Moreover, with slight abuse of notations,
given a $k$-dimensional projection $P$ on $\R^n$, by $|PK|$ we denote the $k$-dimensional volume of $PK$.

The polar set of $K$, denoted by $K^{\circ}$, is defined as
\begin{align*}
K^{\circ}= \{ x \in \R^n \,\,  |\,\, \langle x, y \rangle \leq 1\,\,\, \mbox{ for all } \,\,\, y\in K \}.
\end{align*}

The following result relates the volume of a body with the volume of its polar and is due to Blaschke-Santal\'o and Bourgain-Milman \cite[Theorem 1.5.10 and Theorem 8.2.2]{artstein2015asymptotic}:

{\it There exists an  absolute constant  $c>0$ such that for every centrally symmetric 
convex body  $K\subset \R^n$, }
\begin{align} \label{santalo}
 c |B_2^n|^{2/n} \leq  \vn{K}\vn{K^{\circ}}
\leq |B_2^n|^{2/n}.
\end{align}
In other words,
$ \vn{K}\vn{K^{\circ}}
\sim \frac{1}{n}$.
 We also use the support function of $K$ defined on $\R^n$ by
 $$
    h_K (x) = \sup _{y\in K }  \langle x, y \rangle = \|x\|_{K^{\circ}}.
 $$

Given an operator $T: X\to Y$ the operator norm is denoted by $\|T: X\to Y\|$. Similarly,
given an operator $T: \R^n\to \R^n$, we denote  $\|T\|:= \|T: \ell_2^n \to \ell_2^n\|$.

\bigskip
We now recall some basic properties of the volume ratio, see e.g. \cite{khrabrov2001generalized}.

\begin{fact} \label{propiedades elementales}
For every pair of centrally symmetric convex bodies $(K,L)$ in $\R^n$ the following holds:
\begin{enumerate}
\item \label{prop vr norma}
$$\vr(K,L) = \left(\frac{|K|}{|L|}\right)^{\frac{1}{n}} \cdot \inf_{T \in SL(n,\R)} \Vert T: X_L \to X_K \Vert,$$
where the infimum runs all over the linear transformations $T$ that lie on the special linear group of degree $n$
(matrices of determinant one).

\item $\vr(K,L) \sim  \vr(L^{\circ},K^{\circ})$.
\item If  $T :X_L \to X_K$ is a linear operator we have that
\begin{align*}
\frac{1}{\norm{T: X_L \to X_K}} \cdot T(L) \subset K \quad \mbox{and hence} \quad \vr(K,L) &\leq  \frac{\norm{T: X_L \to X_K}\vn{K}}{\vert \det{T} \vert^{\frac{1}{n}}\vn{L}}.
\end{align*}
\item $\vr(K,L) \leq \vr(K,Z) \cdot \vr(Z,L)$ for every convex body $Z$ in $\R^n$.
\end{enumerate}
\end{fact}

\begin{remark}\label{rem-rs}
We finally discuss not necessarily symmetric convex bodies. 
Note that for every convex bodies $K$ and $L$ in $\R^n$ and
for  any affine transformations $T$ and $S$ one has
$$\vr(K,L)=\vr(T(K),S(L)).$$ In other words, the volume ratio between $K$ and $L$ depends exclusively on the affine
classes of the bodies involved.
By Rogers-Shephards inequality (see e.g., \cite[Theorem~1.5.2]{artstein2015asymptotic}), for every convex body $W \subset \R^n$ we have $\vr(W-W,W) \leq 4$. Clearly the last inequality in
Fact~\ref{propiedades elementales} holds for any (not necessarily centrally symmetric) convex bodies $K, L, Z$.
Therefore,
\begin{align}\label{RSvr}
  \vr(K-K, L) \leq \vr(K-K, K)\, \vr(K,L)\leq   4\, \vr(K,L).
\end{align}
\end{remark}

\section{Auxiliary results.}



We start with  recalling a standard result in geometric measure theory
(see e.g.,  \cite[Theorem 7.5]{mattila1999geometry}).

\index{Hausdorff measure}
\begin{theorem}\label{Matilla}
Let $f :\R^m \to \R^n$ is a Lipschitz map with the Lipschitz constant $L_f$,
$0\leq s\leq m$, and $A \subset \R^m$.  Then
$$\HH^s(f(A)) \leq L_f^s\HH^s(A),$$
where $\HH^s$ is the $s$-Hausdorff measure.
\end{theorem}

Recall that for $k \in \N$ the $k$-Hausdorff measure is a multiple of the Lebesgue measure in $\R^k$. Namely, for every measurable set $A$, $\HH^{k}(A) = \frac{2^k}{|B_2^k|}|A|$.

We denote by $\PP^k(n)$ \index{$\PP^k(n)$}the set of all orthogonal projections of rank $k$ in $\R^n$. Given $Q \in  \PP^k(n)$,
 $|QK|$ denotes the $k$-dimensional Lebesgue measure of $QK$. As an application of the last theorem we prove the
 following lemma, that relates the $k$-dimensional volume of two different projections of $K$ with their distance
 in the canonical operator metric.

\begin{lemma}\label{lema medidas}
 Let $1\le k\leq n$ and let $P,Q \in \PP^k(n)$ be such that $\norm{P-Q} \leq \frac{1}{2\sqrt{n}}$.
Then for every  centrally symmetric convex body $K \subset \R^n$ in the John position,
$$
 \frac{1}{2} |QK|^{\frac{1}{k}} \leq |PK|^{\frac{1}{k}} \leq 2 |QK|^{\frac{1}{k}}.
$$
\end{lemma}

\begin{proof}
Note that
\begin{align*}
P &= P^2 = PQ + P(P -Q).
\end{align*}
Since $K$ is in John's position, $B_2^n \subset K \subset \sqrt{n}B_2^n$.
Using that  $\norm{P-Q} \leq \frac{1}{2\sqrt{n}}$, we observe
$$
  (P -Q)K\subset \sqrt{n} (P -Q) B_2^n \subset \frac{1}{2}\, B_2^n \subset \frac{1}{2}\, K.
$$
This implies
$$
  PK   \subset  PQ\, K + \frac{1}{2}\, PK.
$$
Therefore for every $x\in \R^n$ we have
$$
   h_{PK} (x) \leq h_{PQK} (x) + \frac{1}{2}\,  h_{PK} (x),
$$
so $h_{PK} (x) \leq 2 h_{PQK} (x)$. This means $PK   \subset 2 PQ\, K$.

Finally we  apply Theorem \ref{Matilla} with $m:=n$, $s:=k$, $f:=P$ and $A:=QK$ to obtain
$$
  |PK|^{\frac{1}{k}} \leq 2|PQK|^{\frac{1}{k}} \leq 2 |QK|^{\frac{1}{k}},
$$
using that the Lipschitz constant of the mapping $P$ is obviously one and simplifying the constants
to pass from the Hausdorff to the Lebesgue measure.
\end{proof}

Next we introduce  a variant of Gluskin's random polytopes. Instead of considering the absolute convex hull
of points taken uniformly on the unit sphere we are going to work with Gaussian random vectors.
The reason for doing this is that we want to deal with projections of these bodies,
and the Gaussian measure is more suitable for this purpose.
Let $N > n$ and $g_1, \dots, g_N$ be standard independent Gaussian vectors in $\R^n$.
We consider the symmetric  polytope
$$Z_N = Z_N(\omega) = \absconv\{\sqrt{n}e_1, \dots, \sqrt{n}e_n,g_1,\dots,g_N\}.$$
For  basic properties of Gaussian polytopes we refer  to \cite{mankiewicz2003quotients}.
It is well known that the Euclidean norm of a Gaussian vector in $\R^n$ is well concentrated about it's
average, which is essentially $\sqrt {n}$. We will need the following lemma, the standard proof of which is
provided for the sake of completeness (for simplicity we write just $\|\cdot\|$ for $\|\cdot\|_2$).
\index{Gaussian polytopes}
\begin{lemma}\label{gausnorm}
Let $n\geq 1$ and let $g$ be a standard Gaussian vector  in $\R^n$. Then for every $\lambda \geq 2 \sqrt{n}$,
$$
 \Pro \left\{ \| g \| \geq  \lambda \right\} \leq \exp(- \lambda^2 /8).
$$
In particular,
$$
 \Pro \left\{ \| g \| \geq  2 \sqrt{n} \right\} \leq \exp(- n/2).
$$
Moreover, for $n\geq 50$,
$$
 \Pro \left\{ \| g \| \leq   \sqrt{n}/4 \right\} \leq \exp(- n/4).
$$
\end{lemma}

\begin{proof}
The  Gaussian concentration inequality (see \cite{cirel1976norms}  or inequality (2.35)
in \cite{ledoux2001concentration})  states   for every $s>0$,
$$
  \max\Big\{ \P \left\{ \|g\| - \E \|g\|\geq s\right\}, \, \P \left\{ \E \|g\|- \|g\|)\geq s\right\}\Big\} \leq
   \exp \left({-s^2/2}\right) .
$$
Since, $\E \|g\| \leq (\E \|g\|^2)^{1/2} =\sqrt n$, this yields  the first and the second bounds.
To obtain the third bound, denote $a=\E \|g\|$ and observe
$$
  n-a^2 = \E\left(\|g\| - a\right)^2 =\int^{\infty}_{0} 2t  \P \left\{ |\|g\| - a|\geq t \right\}\, dt
  \leq \int^{\infty}_{0} 4 t e^{-t^2/2} dt= 4 .
$$
Thus $a^2 \geq n-4$ and hence for $n\geq 50$, $a\geq \sqrt{n} (1/4 + 1/\sqrt{2})$.
Applying the concentration inequality with $s =  \sqrt{n/2}$, we obtain
\begin{align*}
  \Pro \left\{ \| g \| \leq  \sqrt{n}/4 \right\} &\leq
  \Pro \left\{ a - \| g \| \geq a- \sqrt{n}/4  \right\} \leq
   \Pro \left\{ a - \| g \| \geq s  \right\}
    \\&
   \leq   \exp \left({-s^2/2}\right) = \exp \left({-n/4}\right),
\end{align*}
which completes the proof.
\end{proof}

\begin{remark}\label{om-zer}
Below we denote
$$
 \Omega_0(n,N):= \{\omega\, \,  |\, \, \forall i\leq N :\, \,
   \sqrt{n}/4 \leq  \| g _i\| \leq   2\sqrt{n} \}.
$$
Lemma~\ref{gausnorm} yields
\begin{align} \label{prob}
    \Pro(\Omega_0(n,N)) \geq 1 - 2 N e^{-n/4}.
\end{align}
Note that on $\Omega_0(n,N)$ we have
\begin{align} \label{om-inc}
  B_2^n \subset Z_N(\omega) \subset 2\sqrt{n} B_2^n .
\end{align}
In fact, Gluskin proved that there exist absolute constants $C, c >0$ such that  for $Cn\leq N\leq e^n$ one has
$$
 c \sqrt{\log  \left(\frac{N}{n}\right)}\,   B_2^n \subset Z_N(\omega) \subset 2\sqrt{n} B_2^n
$$
with probability at least $1-e^{-n}$
(see  \cite[Theorem~2]{gluskin1989deviation} or the remark following the proof of Theorem~2 in \cite{gluskin1989extremal}).
\end{remark}



\medskip

The following theorem establishes a bound for the volume of projections of Gluskin's polytopes.

\begin{theorem} \label{volume estimation}
There exists an absolute constant $C>0$ such that the following holds.
Let $k\leq n$ and $2n\leq N\leq n e^{k}$.
Then there exists a set $\Omega_1(n,N) \subset \Omega_0(n,N)$ such that
for every $\omega \in  \Omega_1(n,N)$ and every $Q \in \mathcal{P}^k(n)$ one has
\begin{align}
 |QZ_N(\omega)|^{1/k} \leq C \max\left\{\frac{\sqrt{n}}{k} \sqrt{\log \log \log \left(\frac{N}{k}\right)}, \,
  \frac{\sqrt{\log(\frac{N}{k})}}{\sqrt{k}}\right\}
\end{align}
  and such that
 $$
    \mathbb{P}(\Omega_1(n,N)) \geq 1- 4 N e^{-n/4}.
 $$
\end{theorem}
To prove the theorem we will need two lemmas.
The first one on the cardinality of $\varepsilon$-nets in $\mathcal P^k(n)$ is due to Szarek \cite{szarek1982nets}.
The second lemma bounds the volume of a polytope in terms of the lengths of the vertices.

\begin{lemma}\label{lemma net}
There exists an absolute positive constant $C_0$ such that for
every $0< \varepsilon <1$ the set $\mathcal P^k(n)$ admits an $\varepsilon$-net
$\Pi$ of cardinality at most $$ \vert \Pi \vert \leq \left(\frac{C_0}{\varepsilon}\right)^{nk}.$$
\end{lemma}

\begin{lemma}\label{lemma boundGnorm}
Let $(w_i)_{i=1}^N \subset \R^n$ be a collection of vectors.
For every $\alpha \geq \sqrt{2} \, \max\limits_{i \leq N} \{\norm{w_i}_2\}$ we have
\begin{align*}
    \vert \absconv\{w_1, \dots, w_N\} \vert^{1/n} \leq \frac{\sqrt{2\pi} e \alpha}{n} \exp \left( \frac{2}{n} \sum_{i=1}^N
    \exp\left(- \frac{\alpha^2}{2\Vert w_i\Vert_2^2}\right) \right).
\end{align*}
\end{lemma}

\begin{remark}
Assuming that $\|w_i\|_2\leq 1$ for every $i\leq N$ and letting
$\alpha = \sqrt{2\log (2N/n) }$ we observe the well known bound (see \cite{barany1988approximation, CP, gluskin1989extremal})
\begin{align}\label{G-C-P}
    \vert \absconv\{w_1, \dots, w_N\} \vert^{1/n} \leq \frac{C \sqrt{\log (2N/n) }}{n} .
\end{align}
Our proof of Lemma~\ref{lemma boundGnorm} follows the proof of this bound with corresponding adjustments.
Note also, that using a standard estimate (\ref{G-C-P}) instead of Lemma~\ref{lemma boundGnorm}, would lead to the bound
\begin{align}
 |QZ_N(\omega)|^{1/k} \leq C \frac{\sqrt{n}}{k} \sqrt{ \log \left(\frac{N}{k}\right)}
\end{align}
in Theorem~\ref{volume estimation}, and thus, to the bound
$$
 \vr(P K, Q L) \geq \frac{c  k}{   \sqrt{   n \log \frac{n\log n}{k}   }   }
$$
in Theorem~\ref{main thm}.
\end{remark}

\begin{proof}[Proof of Lemma~\ref{lemma boundGnorm}]
%
%
%
For simplicity we write $\norm{\cdot}$ for $\norm{\cdot}_2$.
Fix $\alpha \geq \sqrt{2} \, \max\limits_{i \leq N} \{\norm{w_i}\}$ and set  $P_i:=\{x \in \R^n \; : \; |\langle x , w_i \rangle| \leq \alpha\}.$
Consider
$$
 K := \frac{1}{\alpha}\bigcap\limits_{i=1}^N P_i.
$$
 Note that $K^\circ = \absconv\{w_1, \dots, w_N\}$ and that
\begin{align*}
\gamma_n(\alpha K) = \gamma_n \left(\bigcap\limits_{i=1}^N P_i\right) \geq \prod\limits_{i=1}^N \gamma_n(P_i),
\end{align*}
where $\gamma_n$ denotes the Gaussian measure on $\R^n$ and where the inequality follows from
 \v{S}id\'{a}k's lemma (\cite{vsidak1967rectangular},  \cite{gluskin1989extremal}) or from Gaussian correlation inequality (\cite{royen2014simple}, see also \cite{latala2017royen}).

Clearly,
\begin{align*}
    \gamma_n(P_i) = \frac{1}{\sqrt{2\pi}}\int\limits_{-\frac{\alpha}{\norm{w_i}}}^{\frac{\alpha}{\norm{w_i}}}e^{-\frac{t^2}{2}}dt.
\end{align*}
Considering the function
$$
 f(s):= e^{-\frac{s^2}{2}} - \frac{1}{\sqrt{2 \pi}}\int_{-s}^s e^{-\frac{t^2}{2}} dt,
$$
it is not difficult to see that
\begin{align*}
    \frac{1}{\sqrt{2 \pi}}\int_{-s}^s e^{-\frac{t^2}{2}} dt \geq 1 - e^{-\frac{s^2}{2}}.
\end{align*}
Therefore,
\begin{align*}
  \gamma_n(\alpha K) \geq   \prod_{i=1}^N\left(1 - e^{-\frac{\alpha^2}{2\norm{w_i}^2}}\right).
  \end{align*}
Note that, for $x \in (0,\frac{3}{4})$, $1 - x \geq e^{-2x}$. Using that  $\alpha^2 \geq 2 \norm{w_i}^2$ for all $i \leq N$,
we obtain
  \begin{align*}
\gamma_n(\alpha K) \geq \prod_{i=1}^N \exp\left(-2 e^{-\frac{\alpha^2}{2\norm{w_i}^2}}\right) = \exp \left(-2\sum_{i=1}^N e^{-\frac{\alpha^2}{2\norm{w_i}^2}} \right).
  \end{align*}
  Since $|\alpha K| = \alpha^n |K| \geq (2\pi)^{\frac{n}{2}}\gamma_n(\alpha K)$,
  \begin{align*}
      |K|^{\frac{1}{n}} \geq \frac{\sqrt{2 \pi}}{\alpha}\, \exp \left(-\frac{2}{n}\sum_{i=1}^N e^{-\frac{\alpha^2}{2\norm{w_i}^2}} \right).
  \end{align*}
  Finally, the Blashke-Santaló inequality  \eqref{santalo} implies
  \begin{align*}
      |K^\circ|^\frac{1}{n} \leq \frac{|B_2^n|^{\frac{2}{n}}}{|K|^\frac{1}{n}} \leq \frac{2\pi e \alpha}{\sqrt{2\pi}n}\, \exp \left(\frac{2}{n}\sum_{i=1}^N e^{-\frac{\alpha^2}{2\norm{w_i}^2}} \right).
  \end{align*}
  This completes  the proof.
  \end{proof}
We are now ready to prove Theorem \ref{volume estimation}.

\begin{proof}[Proof of Theorem \ref{volume estimation}]

If $k\geq n/16$ the result follows from (\ref{G-C-P}) and  (\ref{prob}), therefore below we assume
$k<n/16$.  Fix $\eps \in  [2\sqrt{k/n}, 1/2]$ to be defined later. Let $C_0$ denote the constant from Lemma~\ref{lemma net}.
Denote $C=40 C_0$ (without loss of generality we assume $C\geq 5$, hence $C\geq 200$) and  set
$$
  m_0= \frac{C k \log (1/\eps)}{4} \quad \mbox{ and } \quad  m_1= \frac{C k \log (1/\eps)}{4\eps^2}=\frac{10 C_0 k \log (1/\eps)}{\eps^2}.
$$
Without loss of generality, we just assume for simplicity that $m_0$ and $m_1$ are integers. Consider the sequence  $(\lambda_m)_{m=1}^N$ defined by
$\lambda_m = 2\sqrt{n}$ for  $i\leq m_0$, $\lambda_m = 2\eps \sqrt{n}$ for  $i> m_1$, and
$$
   \lambda_m = \sqrt{\frac{C n k \log (1/\eps)}{m}} \quad \mbox{ for }\, m_0< i\leq m_1.
$$

Let  $g$ denote a standard Gaussian vector in $\mathbb R^n$. Note that for any fixed projection $Q_0 \in \mathcal P^k(n)$,
 $Q_0 (g)$ is a standard $k$-dimensional Gaussian vector.
Thus, by Lemma~\ref{gausnorm}, for every $t\geq 2\sqrt{k}$ we have
\begin{align}\label{gaussian inequality}
    \mathbb{P} \left\{ \omega \in \Omega\,\, |\,\, \Vert Q_0(g(\omega)) \Vert \geq t \right\} \leq e^{-t^2 /8}.
\end{align}

Let $g_1, \dots, g_N$ be standard Gaussian independent vectors in $\R^n$. For a fixed projection $Q_0$ in $\mathcal P^k(n)$ and for
$m\leq N$ consider the events
$$
 \mathbb A(m, Q_0) := \Big\{ \omega \in \Omega_0(n,N)\,\, | \, \,
    \#\{i : \Vert Q_0(g_i(\omega)) \Vert > \lambda_m \} \geq  m \Big\}.
$$
Note that on $\Omega_0(n,N)$ we have $\Vert g_i(\omega) \Vert\leq 2\sqrt{n}$, hence
$A(m, Q_0)= \emptyset $ for $m\leq m_0$. By $ A_{Q_0}$ denote the union (over $m$) of $A(m, Q_0)$,
that is
\begin{align*}
 \mathbb A_{Q_0} &= \Big\{ \omega \in \Omega_0(n,N)\,\, | \, \,   \exists  m \in \{1, \dots, N\} :\, \,
   \#\{i : \Vert Q_0(g_i(\omega)) \Vert > \lambda_m \} \geq m \Big\} \\&
   =\Big\{ \omega \in \Omega_0(n,N)\,\, | \, \,   \exists  m \in \{m_0+1, \dots, N\} :\, \,
   \#\{i : \Vert Q_0(g_i(\omega)) \Vert > \lambda_m \} \geq m \Big\}.
\end{align*}
To estimate the probability of $ A_{Q_0}$ we first note that
\begin{align}\label{necbound}
 \frac{eN}{m_1} \, e^{-\eps^2 n /2}\leq \frac{eN}{m_1} \, e^{-\eps^2 n /4} \leq \frac{1}{2}.
\end{align}
Indeed, consider the function
$$
 f(\eps):=\eps^{-2}  \, e^{\eps^2 n /4}.
$$
Since it is increasing on $[2/\sqrt{n}, \infty)$, using that  $\eps \geq 2\sqrt{k/n} \geq 2/\sqrt{n}$,
we observe that
$$
  f(\eps)\geq \frac{n}{4 k}  e^{k}.
$$
Thus, using that $\eps<1/2$, $N\leq n e^{k}$ and $C\geq 200$,
$$
   \frac{eN}{m_1} \, e^{-\eps^2 n /2}= \frac{4eN \eps^ 2 }{Ck\log(1/\eps)}\, e^{-\eps^2 n /2} \leq \frac{4 e n e^{k}}{Ck(\log 2) f(\eps)}
   \leq \frac{16 e   }{C\log 2}  \leq  \frac{1}{2}.
$$

Denote
$$
  p:= \exp({- C_0 nk \log (1/\eps)}).
$$
Using the union bound, the independence of $g_i$'s,   equations \eqref{gaussian inequality}, \eqref{necbound}, and
the standard bound
$$
 \sum_{m=0}^{\ell}  \binom{N}{m} \leq   \left(\frac{eN}{\ell}\right)^{\ell},
$$
  we obtain
\begin{align}
   \nonumber{P}(\mathbb {A}_{Q_0})  &\leq \sum_{m=m_0+1}^N  \mathbb{P}\left(\mathbb {A} (m, Q_0) \right) \leq \sum_{m=m_0+1}^N  \binom{N}{m}  e^{-\lambda_m^2 m/8} \\& \nonumber
\leq \sum_{m=m_0+1}^{m_1}  \binom{N}{m}  \exp({-C nk \log (1/\eps) /8}) + \sum_{m=m_1+1}^N  \binom{N}{m}  e^{-\eps^2 n m/2}\\&\nonumber
\leq \left(\frac{eN}{m_1}\right)^{m_1} p^5 + \sum_{m=m_1+1}^N  \left(\frac{eN}{m} e^{-\eps^2 n /2}\right)^{m}  \\&\nonumber
\leq \left(\frac{eN}{m_1}\right)^{m_1} p^5 + \left(\frac{eN}{m_1} e^{-\eps^2 n /2}\right)^{m_1} \sum_{m=1}^\infty  \left(\frac{1}{2} \right)^{m} \\&\nonumber
 =
 \left(\frac{eN}{m_1}\right)^{m_1} \left(p^5 + \exp({-C nk \log(1/\eps)/8})\right) \leq 2 \left(\frac{eN}{m_1}\right)^{m_1} p^5.
\end{align}
 Using \eqref{necbound} again, we estimate
$$
   \left(\frac{eN}{m_1}\right)^{m_1} \leq \exp\left(\frac{\eps^2 n m_1}{4}\right)   \leq
   \exp\left(  \frac{10 C_0 k n\log (1/\eps)}{4}\right)=    p^{- 2.5}.
$$
Hence,
\begin{align} \label{acotacion prob}
   \mathbb{P}(\mathbb {A}_{Q_0})  &\leq  2p^2.
\end{align}

For each $\omega \in \Omega_0(n,N)$ define the vector $a= a(\omega) \in \mathbb R^N$
by $a_i := \Vert Q_0(g_i(\omega)) \Vert$ for $ i \leq N$. Then, by definition,
 on $\Omega_0(n,N)\cap \mathbb{A}_{Q_0}^c$ we have
$$a^{*}_m \leq \lambda_m,$$
where $a^{*}$ stands for the decreasing rearrangement of  $a$.

Similarly, given $Q \in \PP^k(n)$ for each $\omega \in \Omega$ define the vector $b= b(\omega, Q) \in \mathbb R^N$
by $b_i := \Vert Q(g_i(\omega)) \Vert$ for $i \leq N$.
Let
$$\mathbb{B} := \left\{ \omega \in  \Omega_0(n,N)\,\, | \, \, \exists Q \in \PP^k(n)\,     \exists m\leq N : \,\,  b^*_m(\omega, Q) > 2\lambda_m  \right\}.$$


We now use an approximation argument. Let $Q \in \mathcal P^{k}(n)$ and consider $Q_0$ such that $\Vert Q - Q_0\Vert < \varepsilon.$
Then,
\begin{align*}
    \Vert Q(g_i(\omega)) \Vert & \leq \Vert Q_0(g_i(\omega)) \Vert + \Vert Q - Q_0 \Vert \Vert g_i(\omega) \Vert \\ & \leq  \Vert Q_0(g_i(\omega)) \Vert + \varepsilon \max\norm{g_i(\omega)}_2.
\end{align*}
Therefore, for $\omega \in \Omega_0(n,N)\cap  \mathbb{A}_{Q_0}^c$ and for every $m\leq N$ we have
$$
 b^*_m(\omega, Q) \leq a^*_m + 2\eps \sqrt{n} \leq \lambda_m + 2\eps \sqrt{n} \leq 2 \lambda_m.
$$
Let $\Pi \subset \PP^k(n)$ be an $\varepsilon$-net  of cardinality at most  $\left(\frac{C_0}{\varepsilon}\right)^{nk}\leq\frac{1}{p}$ given by Lemma~\ref{lemma net}. Then,
$$
 \mathbb{B} \subset   \bigcup_{Q_0 \in \Pi} A_{Q_0},
$$
and therefore, by (\ref{acotacion prob}),
\begin{align*}
    \Pro(\mathbb{B}) \leq 2p .
\end{align*}
Therefore, defining  $\Omega_1(n,N):= \mathbb{B}^c \cap \Omega_0(n,N)$,
 by (\ref{prob}), we obtain
\begin{align*}
    \Pro(\Omega_1(n,N)) \geq 1-2N e^{-n/4} - 2p \geq 1- 4 N e^{-n/4} .
\end{align*}

It remains to estimate volumes of corresponding polytopes for $\omega \in \Omega_1(n,N)$. They can be written
as $Q(Z_N(\omega))=\absconv\{w_1, \dots, w_N\}$ with $\|w_m \|\leq 2\lambda_m$ for every $m\leq N$. We first estimate
\begin{align*}
  A:&=   \sum_{m = 1}^N  e^{-\frac{\alpha^2}{8 \lambda_m^2}} \leq
  m_0 e^{-\frac{\alpha^2}{32 n}} + \sum_{m= m_0 + 1}^{m_1} e^{-\frac{\alpha^2 m}{8 C n k \log\left(\frac{1}{\varepsilon}\right)}}  + (N-m_1)  e^{-\frac{\alpha^2}{32\varepsilon^2 n}}   \\&
    \leq   m_0 e^{-\frac{\alpha^2}{32 n}} +\left(1-e^{-\frac{\alpha^2 }{8 C n k \log\left(\frac{1}{\varepsilon}\right)}}\right)^{-1}  e^{-\frac{\alpha^2 (m_0+1)}{8 C n k \log\left(\frac{1}{\varepsilon}\right)}}  +  Ne^{- \frac{\alpha^2}{32 \varepsilon^2 n }} \\&
     \leq   \left( m_0 + \max\left\{2, \frac{16 C n k \log\left(\frac{1}{\varepsilon}\right)}{\alpha^2 }\right\}\right) e^{-\frac{\alpha^2}{32 n}}  +  Ne^{- \frac{\alpha^2}{32 \varepsilon^2 n }} \\&
    =  \frac{C k \log (1/\eps)}{4}\, \left(1 +\max\left\{1, \frac{16  n }{\alpha^2 }\right\}\right)\, e^{-\frac{\alpha^2}{32 n}}  +  Ne^{- \frac{\alpha^2}{32 \varepsilon^2 n }},
\end{align*}
where we used that $e^{-x}\leq \max\{1-x/2, 1/2\}$ for $x>0$.
We choose
$$\alpha = 6\sqrt{n} \max\left\{ \sqrt{\log\Big( C \log \frac{1}{\varepsilon}\Big)}, \sqrt{\log\frac{N}{k}} \cdot \varepsilon\right\},$$
then  $\frac{2}{k} A \leq 2$. Furthermore, we choose
$$
 \varepsilon =\max\left\{ \frac{\sqrt{\log\log\log\frac{N}{k}}}{\sqrt{\log\frac{N}{k}}}, 2 \sqrt{\frac{k}{n}}\right\}
$$
(recall that $k\leq n/16\leq N/32$), then
$$
   \alpha \leq C_1 \max\left\{  \sqrt{n \, \log\log\log\frac{N}{k}} ,\,  \,  2 \sqrt{k\, \log\frac{N}{k}}\right\},
$$
where $C_1>0$ is an absolute constant. Applying  Lemma~\ref{lemma boundGnorm} for $Q(Z_N(\omega))$ (note that
$Q(Z_N(\omega))$ is $k$-dimensional) we obtain
$$
  |Q(Z_N(\omega))|^{1/k} \leq \frac{\sqrt{2\pi} e^3 \alpha}{k} \leq C_2
  \max\left\{ \frac{\sqrt{n}}{k} \, \sqrt{\log\log\log\frac{N}{k}} , \frac{\sqrt{\log\frac{N}{k}}}{ \sqrt{k}}\r\},
$$
where $C_2>0$ is an absolute constant.
This completes the proof.
\end{proof}

\section{Proof of the main theorem.}



We first  prove a series of lemmas.
Given  two $k$-dimensional subspaces of $\R^n$,  $E$ and $F$, we denote by $\SS(E, F)$
the set of all linear operators
$T:E \to F$ preserving the volume (the $k$-dimensional Lebesque measure).
If $E=Q_0 \R^n$ and $F=Q_1 \R^n$ for some $Q_0, Q_1 \in \PP^k(n)$ we simply write $\SS(Q_0, Q_1)$.

\begin{lemma}\label{lemavol projections}
Let $K \subset \R^n$ be a centrally symmetric convex body,
$Q_0, Q_1 \in \PP^k(n)$ be  fixed orthogonal projections of rank $k$, and $A>0$.
Let  $T_0 \in \SS(Q_0, Q_1)$ be a fixed linear operator.
Then
$$
 \Pro\{\omega \in \Omega\,\, |\, \,\norm{T_0 : X_{Q_0Z_N(\omega)} \to X_{Q_1 K}} \leq A\} \leq (A/\sqrt{2\pi})^{kN} |Q_1 K|^N.
$$
\end{lemma}
\begin{proof}
Observe that
\begin{align*}
 T_0Q_0(Z_N(\omega)) \subset AQ_1 K &\Longleftrightarrow
 Q_0Z_N(\omega) \subset AT_0^{-1}(Q_1 K) \\&
 \Longleftrightarrow \forall i \leq N \, : \, \,
  Q_0 g_i(\omega) \in AT_0^{-1}(Q_1 K) .
\end{align*}
Note that for every $k$-dimensional convex body $L$ one has $\gamma_k(L) \leq (2\pi)^{-k/2}  |L|$.
Using this, the rotational invariance of the Gaussian measure, and the fact that $T_0$ preserves the
Lebesgue  measure in $Q_0 \R^n$, we observe that for every $i\leq N$,
$$
 \Pro\{\omega \in \Omega\,\, |\, \,Q_0 g_i(\omega) \in AT_0^{-1}(Q_1 K)\} \leq (2\pi)^{-k/2} |AT_0^{-1}(Q_1 K)| =
 (2\pi)^{-k/2} A^k |Q_1 K| .
$$
The result follows by the indepenence of $g_i$'s.
\end{proof}

\begin{lemma}\label{lemaprobQ0fijo}
There exists and absolute constant $C>0$ such that the following holds.
Let $k\leq n$,  $Cn\leq N\leq n e^{k}$, and $A>0$.
Let $K \subset \R^n$ be a centrally symmetric convex body and
$Q_0, Q_1 \in \PP^k(n)$ be  fixed orthogonal projections of rank $k$.
Then
\begin{align*} \Pro \{ \omega \in \Omega_0(n,N)\,\, |\,\, \exists T \in \SS(Q_0, Q_1) &  \,\,
\mbox{ such that } \,\,  \norm{T: X_{Q_0Z_N(\omega)} \to X_{Q_1 K}} \leq A \} \\
& \leq (5\sqrt{n})^{k^2}  A^{Nk}|Q_1 K|^N.
\end{align*}
\end{lemma}
\begin{proof}
Let $E:= \ell_2^n \cap Q_0\R^n$ and let  $U:=B_{\LL(E, X_{Q_1 K})}$ be the unit ball of $\LL(E, X_{Q_1 K})$. Denote
$$
  a:= \frac{A}{2  \sqrt{n }} .
$$
Let $\NN$ be a maximal  $a$-separated set in $A U \cap \SS(Q_0, Q_1)$
in the metric $\Vert \cdot  \Vert_{\LL(E, X_{Q_1 K})}$. By the maximality of
$\NN$, the set $\NN$ is an $a$-net for $A U \cap \SS(Q_0, Q_1)$
and moreover,  the following inclusion for the disjoint union holds,
$$
  \bigcup_{\eta \in \NN} \left(\eta + \frac{a }{2}\, U \r) \subset \left(A + \frac{a}{2}\r)\, U.
$$
Identifying the space with $\R^{k^2}$ and computing volumes we conclude that
$$\#\NN \leq \left(\frac{A+a/2}{a/2}\r)^{k^2}\leq  (5\sqrt{n})^{k^2}.$$

Take $\omega \in \Omega_0(n,N)$ such that there exists $T \in \SS(Q_0, Q_1)$ with
$$
   \norm{T : X_{Q_0Z_N(\omega)} \to X_{Q_1 K}} \leq A.
$$
Recall that by Remark~\ref{om-zer} we have on $\Omega_0(n,N)$,
$$
 B_2^n \subset Z_N(\omega)\subset 2\sqrt{n} B_2^n ,
$$
hence
 $T \in A U$. Since $\NN$ is an $a$-net for $A U \cap \SS(Q_0, Q_1)$ there is $S \in \NN$ such that
 $$\norm{S-T : E \to X_{Q_1 K}} \leq a.$$
Using that, $Z_N(\omega)\subset 2\sqrt{n}B_2^n$,
\begin{align*}
\norm{S : X_{Q_0 Z_N(\omega)} \to X_{Q_1 K}} &\leq \norm{S - T : X_{Q_0Z_N(\omega)} \to X_{Q_1 K}} + \norm{T : X_{Q_0Z_N(\omega)} \to X_{Q_1 K}}\\
&\leq 2\sqrt{n}\, \norm{S - T : E \to X_{Q_1 K}} + A \leq 2\sqrt{n} a + A= 2A.
\end{align*}
This shows
\begin{align*}
 \{ \omega \in \Omega_0(n,N) \,\, |\,\,  \exists T \in \SS(Q_0, Q_1)  \,\, \mbox{ such that } \,\,
 \norm{T: X_{Q_0Z_N(\omega)} \to X_{Q_1 K}} \leq A \}\\ \subset \bigcup_{S \in \NN} \{S\,\, |\,\,
 \norm{S: X_{Q_0Z_N(\omega)} \to X_{Q_1 K}} \leq 2A\}.
\end{align*}
Using the union bound and applying Lemma~\ref{lemavol projections}, we obtain the desired bound.
\end{proof}

Given bases $B = \{v_1, \dots, v_k \}$ and $B' = \{v_1', \dots, v_k'\}$ of  vector spaces $F$ and $F'$ and a vector $x \in F$ we denote by $(x)_B$ the coordinates of $x$ in the basis $B$ (similarly, $(y)_{B'}$ for $y\in F'$). That is, $(x)_B=(\alpha_1, \dots, \alpha_k)$ if $x=\sum_{i=1}^k \alpha_i v_i.$
Also for an operator $T:F \to F'$ we denote by $[T]_{B,B'}$ the matrix $(a_{i,j})_{1 \leq i,j \leq k}$ such that
$T(v_\ell)=\sum_{i=1}^k a_{i,\ell} v_i'$, for every $1 \leq \ell \leq k$ (i.e., the $\ell$-column of $[T]_{B,B'}$ is  $(Tv_\ell)_{B'}^t$).

\begin{lemma}\label{lemaprobQT}
Let $K \subset \R^n$ be a centrally symmetric convex body  in John's position. Then for every $\beta>0$ one has
\begin{align*}
\Pro \Bigg\{ \omega \in \Omega_0(n,N)\, |\, & \exists Q_0, Q_1 \in \PP^k(n)\,  \exists T \in \SS(Q_0, Q_1)  \, :\,\,
 \norm{T: X_{Q_0 Z_N(\omega)} \to X_{Q_1 K}} \leq \frac{\beta}{|Q_1 K|^{\frac{1}{k}}}\Bigg\} \\ &
\leq C^{kN} ( \sqrt{n}) ^{2 nk+k^2}  {\beta}^{kN},
\end{align*}
where $C>0$ is an absolute constant.
\end{lemma}
\begin{proof}
By  Lemma~\ref{lemma net} there is  a $\frac{1}{2\sqrt{n}}$-net, say $\Pi$, for $\PP^k(n)$ of cardinality $\#\Pi \leq (C_0\sqrt{n})^{nk}$.
By Lemma~\ref{lemaprobQ0fijo} and the union bound, it is enough to show that
\begin{align*}
 \Bigg\{& \omega \in \Omega_0(n,N)\, |\,  \exists Q_0, Q_1 \in \PP^k(n),\, \;  \exists T \in \SS(Q_0, Q_1)  \, :\,\,
 \norm{T: X_{Q_0 Z_N(\omega)} \to X_{Q_1 K}} \leq \frac{\beta}{|Q_1 K|^{\frac{1}{k}}}\Bigg\} \\ &
\subset \bigcup_{Q_0', Q_1' \in \Pi} \left\{ \omega \in \Omega_0(n,N) \, | \,  \exists S  \in \SS(Q_0', Q_1') \, :
\,\, \norm{S: X_{Q_0'Z_N(\omega)} \to X_{Q_1'K}} \leq C' \frac{\beta}{|Q_1' K|^{\frac{1}{k}}}\right\}.
\end{align*}

Let $\omega \in \Omega_0(n,N)$ be such that there are $Q \in \PP^k(n)$ and $T \in \SS(Q_0, Q_1)$ with
\begin{align}\label{ecuacion00}
 \norm{T: X_{Q_0 Z_N(\omega)} \to X_{Q_1 K}} \leq \frac{\beta}{|Q_1K|^{\frac{1}{k}}}.
\end{align}
Take $Q_0', Q_1' \in \Pi$ such that $\norm{Q_i - Q_i'} \leq \frac{1}{2\sqrt{n}}$, $i=1,2$.
Fix orthonormal bases
$$B= \{v_1,\dots,v_k\} \quad \quad \mbox{ and }\quad \quad B'= \{v_1',\dots,v_k'\}$$
 of $Q_0\R^n$ and $Q_1\R^n$ respectively.
It is easy to see that the collections
$$
  B_0 = \{Q_0'v_1,\dots,Q_0'v_k\}\quad \quad \mbox{ and }\quad \quad B_1 = \{Q_1'v_1',\dots,Q_1'v_k'\}
$$
are bases of $Q_0'\R^n$ and  $Q_1'\R^n$ respectively. Let $S : Q_0'\R^n \to Q_1'\R^n$ be such that
$$[S]_{B_0, B_1} = [T]_{B, B'},$$
in particular,
 $S \in \mathcal S(Q_0', Q_1').$

\smallskip

It is enough to show that
\begin{align*}
\norm{S: X_{Q_0' Z_N(\omega)} \to X_{Q_1' K}} \leq \frac{ C'\, \beta}{|Q_1' K|^{\frac{1}{k}}},
\end{align*}
which by  Lemma~\ref{lema medidas} reduces to
\begin{align*}
\norm{S: X_{Q_0' Z_N(\omega)} \to X_{Q_1' K}} \leq \frac{ C'\, \beta}{2 |Q_1 K|^{\frac{1}{k}}}.
\end{align*}
Take $x \in Z_N(\omega)$ and note
$$
  SQ_0' x= \underbrace{SQ_0' (Q_0' x - Q_0 x)}_{(1)} + \underbrace{SQ_0' Q_0 x}_{(2)}.
$$
We check that both terms $(1)$ and $(2)$ are contained in a multiple of $\frac{\beta}{|Q_1 K|^{\frac{1}{k}}} Q_1' K$.

\smallskip

We start with the second term,  $SQ_0' Q_0 x$. Write $Q_0 x = \sum \alpha_iv_i$, so $Q_0' Q_0 x = \sum \alpha_iQ_0' v_i$.
We have,
\begin{align*}
(SQ_0' Q_0 x)_{{B_1}}^t &= [T]_{B, B'} (Q_0' Q_0 x)_{{B_0}}^t\\
&= [T]_{B, B'} (Q_0 x)_B^t\\
&=(TQ_0 x)_{B'}^t.
\end{align*}
Therefore $S Q_0' Q_0 x = Q_1' T Q_0 x$. Since $x \in Z_N(\omega)$, by (\ref{ecuacion00}) we have
$T Q_0 x \in \frac{\beta}{|Q_1 K|^{\frac{1}{k}}} Q_1 K$, hence
$$
   Q_1' T Q_0 x \in \frac{\beta}{|Q_1 K|^{\frac{1}{k}}} Q_1' Q_1 K.
$$
Since $K$ is in John's position,
 $B_2^n \subset K \subset \sqrt{n}B_2^n$. Using $\norm{Q_1 - Q_1'} \leq \frac{1}{2 \sqrt{n}}$, we obtain
\begin{align}\label{QQK}
Q_1' Q_1 K \subset& \, Q_1' K + Q_1' ((Q_1 - Q_1')K)\\ \nonumber
\subset& \, Q_1' K + Q_1' ((Q_1-Q_1')\sqrt{n} B_2^n)\\  \nonumber
\subset& \, Q_1' K + Q_1' B_2^n \\  \nonumber
\subset & \, 2Q_1' K.
\end{align}
This implies
$$
S Q_0' Q_0 x =  Q_1' T Q_0 x  \in  \frac{\beta}{|Q_1 K|^{\frac{1}{k}}} Q_1' Q_1 K
\subset \frac{2\beta}{|Q_1 K|^{\frac{1}{k}}} Q_1'  K ,
$$
proving the inclusion for term  $(2)$.

\smallskip

Next we deal with the first term, $SQ_0'(Q_0' x - Q_0 x)$.
Recall that  $Z_N(\omega) \subset 2 \sqrt{n}B_2^n$ on $\Omega_0(n,N)$ and $\norm{Q_0 - Q_0'} \leq \frac{1}{2 \sqrt{n}}$.
Therefore $(Q_0' - Q_0)x \in  B_2^n$ and then $Q_0' (Q_0' - Q_0)x \in B_2^k$.
Thus, it is enough to show  that
$$
  \norm{S : X_{Q_0' B_2^n } \to X_{Q_1' K}}\leq C''\frac{\beta}{|Q_1 K|^{\frac{1}{k}}}.
$$

Take $ y \in Q_0' \R^n$ with $\norm{y}_2 = 1$. Write $(y)_{{B_0}} = (\beta_1,\dots,\beta_k)$.
Then,
\begin{align}\label{gam}
(\gamma_1,\dots,\gamma_k) := (Sy)_{{B_1}} =& [T]_{B, B'} (\beta_1,\dots,\beta_k)^t
= (T(\sum \beta_iv_i))_{B'}.
\end{align}
Notice that
\begin{align*}
\norm{\sum \beta_i v_i}_2 \leq& \norm{\sum \beta_i Q_0 v_i - \sum \beta_iQ_0' v_i}_2 + \norm{\sum \beta_i Q_0' v_i}_2 \\
\leq& \frac{1}{2 \sqrt{n}} \norm{\sum \beta_i v_i}_2 + 1,
\end{align*}
which implies
\begin{align*}
\norm{\sum \beta_iv_i}_2 &\leq \frac{1}{1-\frac{1}{2\sqrt{n}}} \leq 2.
\end{align*}
Since by  (\ref{om-inc}) we have $B_2^n\subset Z_N(\omega)$ on $\Omega_0(n,N)$, using
 (\ref{ecuacion00}), we observe
\begin{align}\label{tbeta}
 T (\sum \beta_i v_i) &\in \frac{2\beta}{ |Q_1 K|^{\frac{1}{k}}}Q_1 K.
\end{align}
On the other hand, by (\ref{gam}),
\begin{align*}
Sy =& \sum \gamma_i Q_1' v_i' = Q_1' (\sum \gamma_i v_i') = Q_1' T(\sum \beta_i v_i).
\end{align*}
Therefore, using (\ref{tbeta}) and (\ref{QQK}),
$$
  Sy \in \frac{2\beta}{ |Q_1 K|^{\frac{1}{k}}}Q_1' Q_1 K \subset \frac{4\beta}{ |Q_1 K|^{\frac{1}{k}}}Q_1 K.
$$
This completes the proof.
\end{proof}

We are now ready to prove our main result.

\begin{proof}[Proof of Theorem \ref{main thm}. ]
First note that it is enough to consider orthogonal projections only.
Indeed, let $P$ be a projection of rank $k$ and $Q$ be the orthogonal
projection with the same kernel as $P$. Then $QP =Q$, hence $Q(PK)=QK$ and
 $\vr(PK, W)=\vr(QK, W)$ for every convex body $W$.

Furthermore, by (\ref{RSvr}) for every $Q_0, Q_1 \in \PP^k(n)$ we have
$$
  \vr(Q_0(K-K), Q_1 Z) \leq \vr(Q_0(K-K),Q_0 K) \vr(Q_0 K, Q_1 Z) \leq 4 \vr(Q_0 K, Q_1 Z).
$$
Thus it is enough to estimate from below  $\vr(Q_0 (K-K), Q_1 Z)$, in other words without loss of generality, we may assume
 that $K$ is centrally symmetric. Since volume ratio is an affine invariant, we may also assume that
$K$ is in John's position.

Let $\beta >0$ and $n<N\leq n \log n$.  Let $\Omega_1(n,N)$ be the set  given by Theorem~\ref{volume estimation} (note that for $k\leq \sqrt{n}$
the result is trivial, so we may assume that  $k\geq \sqrt{n}$ and thus the assumption on $N$ in  Theorem~\ref{volume estimation}
is satisfied).
 Consider  the event
\begin{align*}
\mathcal{E}_\beta:= \Big\{ \exists Q_0, Q_1 \in \PP^k(n), & \, \;\exists T \in \mathcal S (Q_0, Q_1)\, :\,\,  \norm{T: X_{Q_0 Z_N(\omega)} \to X_{Q_1 K}} \leq \frac{\beta}{|Q_1 K|^{\frac{1}{k}}}\Big\}.
\end{align*}
Since $\Omega_1(n,N) \subset \Omega_0(n,N)$,  Lemma \ref{lemaprobQT} yields  that for some absolute constant $C>0$,
\begin{align*}
\nonumber \Pro \left( \mathcal{E}_\beta \bigcap \Omega_1(n,N)\right)
 \leq  C^{kN} ( \sqrt{n}) ^{2nk+k^2}  {\beta}^{kN} \leq (C\beta) ^{kN} \, n^{2nk}  .
\end{align*}
Choose  $N = n\log(n)$ and $\beta= C^{-1} e^{-3}$. Then, using Theorem~\ref{volume estimation} we obtain
$$
 \Pro \left( \mathcal{E}_\beta \right) \leq e^{-Nk} + \Pro \left( \Omega_1(n,N)\right) \leq e^{-Nk} +
 4 N e^{-n/4}\leq 5 n \log (n) \, e^{-n/4} <1.
$$
Thus there is $\omega \in \Omega_1(n,n\log(n))$ such that for every $Q_0, Q_1 \in \PP^k(n)$ and
$T \in \mathcal S (Q_0, Q_1)$,
$$
  \norm{T: X_{Q_0 Z_N(\omega)} \to X_{Q_1 K}} \geq \frac{\beta}{|Q_1 K|^{\frac{1}{k}}}.
$$
Using Fact~\ref{propiedades elementales} \eqref{prop vr norma} and Theorem~\ref{volume estimation}, we conclude that
\begin{align*}
\vr(Q_1 K, Q_0 Z) &\geq \frac{|Q_1 K|^\frac{1}{k}}{|Q_0 Z|^\frac{1}{k}}\,\, \frac{\beta}{|Q_1 K|^{\frac{1}{k}}}\\
\geq &  C' \beta \,
\min\left\{\frac{k}{\sqrt{n} \sqrt{\log \log \log \left(\frac{N}{k}\right)}}, \,
  \frac{\sqrt{k}}{\sqrt{\log(\frac{N}{k})}}\r\}
\end{align*}
which proves the desired result.
\end{proof}

Finally, for the sake of completeness, we prove Corollary \ref{dual thm}.

\begin{proof}[Proof of Corollary \ref{dual thm}]

Let $E, F$ be $k$-dimensional  subspaces  $\R^n$.
Applying Theorem~\ref{main thm} and Fact~\ref{propiedades elementales} (2) for $K^\circ$,
there is a centrally symmetric body $W$ such that
\begin{align*}
\min\left\{\frac{  k}{  \sqrt{n}} \cdot \sqrt{\frac{1}{\log \log \log(\frac{n\log(n)}{k})}}, \,
  \frac{\sqrt{k}}{\sqrt{\log(\frac{n\log(n)}{k})}}\right\} & \lesssim \vr\left(P_E K^\circ,P_F W\right) \\
   &\sim \vr\left((P_F W)^\circ, (P_E K^\circ)^\circ\right)\\
   &= \vr\left(F \cap W^\circ,E \cap K\right),
\end{align*}
where we used that $(P_E K^\circ)^\circ = E \cap K$ and $(P_F W)^\circ = E \cap W^\circ$.
This completes the proof.
\end{proof}

\section{Sharpness.}

We will use the following Rudelson's result proved in \cite{rudelson2004extremal} (see the first page of
the Section~4 in that paper, p. 1077).

\begin{theorem}\label{Rudel}
  Let $1\leq k\leq n/16$. Let $L\subset \R^n$ be a centrally symmetric convex  body.
  Then there are a parameter
  $t=t(L)$ and a linear operator $T:\R^n\to \R^n$ of rank  $k$ such that
  $$
    c\left(B_1^k + \frac{t}{\sqrt{n} \, \log n} B_2^k \right) \subset TL \subset
    C\left(t B_1^k + \sqrt{\frac{k}{n} } B_2^k \right)\subset
    C\left(t  +\frac{k}{\sqrt{n} } \right)B_1^k ,
  $$
  where $C>c>0$ are absolute constants.
\end{theorem}

As a consequence of Rudelson's theorem we obtain the following bound.

\begin{corollary}\label{sharp}
  Let $1\leq k\leq n$. Let $L\subset \R^n$ be a  convex  body.
  Then there is  a $k$-dimensional projection $Q$ such that
  $$
    \vr(B_1^k, QL) \leq C \max\left\{ \frac{k}{\sqrt{n}},\,  \sqrt{\frac{n}{k}} \log n\right\},
  $$
   where $C>0$ is an absolute constant.
\end{corollary}

Before providing a proof of the previous corollary we make state two important remarks.

\begin{remark}\label{rem-l-1}
Recall that by (\ref{GH-vr}) and Remark~\ref{rem-rs}, for every convex body $L\subset \R^n$ and
every projection $Q$ of rank $k$,
\begin{align} \label{pvr-b-1}
\vr(B_1^k, QL) &\leq  \vr(B_1^k, QL-QL) \vr(QL-QL, QL)  \leq 4 \sqrt{{ek}}.
\end{align}
Thus
Corollary~\ref{sharp} implies that for every  convex  body $L\subset \R^n$
 there exists  a $k$-dimensional projection $Q$ such that
$$
      \vr(B_1^k, QL) \leq C
 \begin{cases}
    \frac{k}{\sqrt{n}}  & \text{if } k\geq n^{\nicefrac{2}{3}}  (\log n)^{\nicefrac{2}{3}}\\
    \sqrt{\frac{n}{k}} \log n &\text{if } \sqrt{n} \log n< k\leq n^{\nicefrac{2}{3}}  (\log n)^{\nicefrac{2}{3}}\\
    \sqrt{k} &\text{if }  k \leq  \sqrt{n} \log n.
\end{cases}
$$
\end{remark}

\begin{remark}
 Clearly, $B_1^k$ can be realized as a (coordinate) projection of $K=B_1^n$.
 Thus Corollary~\ref{sharp} shows sharpness of Theorem~\ref{main thm} (up to
 logarithmic factors) in the regime $k\geq n^{2/3}$.
 Note that Rudelson's bound (\ref{Rudbou}) has the same phase transition $k\sim n^{2/3}$.
\end{remark}

\begin{proof}[Proof of Corollary~\ref{sharp}.]
Clearly we may assume that $k\leq n/16$ (otherwise we may use (\ref{pvr-b-1})).

Using again Remark~\ref{rem-rs} implying that for
every projection $Q$ of rank $k$,
$$\vr(B_1^k, QL) \leq  \vr(B_1^k, QL-QL) \vr(QL-QL, QL) \leq 4 \vr(B_1^k, QL-QL),$$
without lost of generality we assume that $L$ is centrally symmetric.

Let $T$ be a projection given by Theorem~\ref{Rudel}.
We consider two cases. First assume that $t(L)\leq k/\sqrt{n}$.
In this case
$$
  c\,   B_1^k   \subset   TL \subset  2C\, \frac{k}{\sqrt{n}}  B_1^k.
$$
  This implies
$$
   \vr(B_1^k, TL) \leq \frac{2C\,k}{c \sqrt{n}} .
$$
  The second case is $t(L)> k/\sqrt{n}$. Using again $B_2^k\subset \sqrt{k} B_1^k$,
in this case we have
   $$
     \frac{c \, t}{\sqrt{n} \, \log n} B_2^n  \subset TL \subset
    2 t C B_1^k  .
  $$
  This implies
$$
   \vr(B_1^k, TL) \leq  \frac{2C\,\sqrt{n} \, \log n}{c }
   \left(\frac{|B_1^k|}{|B_2^k|}\right)^{1/k}\leq \frac{2C\,\sqrt{n} \, \log n}{c \sqrt{k}} .
$$
Finally, note that, similarly to the beginning of the proof of Theorem~\ref{main thm},
any image of a convex body under a linear operator of rank $k$ is
on the Banach--Mazur distance 1 to an image of the body under a projection having the same kernel.
Since volume ratio is an affine invariant,  this completes the proof.
\end{proof}

\section{Concluding remarks.}

In fact our results can be interpreted in terms of the following parameter
of convex bodies.

Let $1\leq k\leq n$, for a given convex body $K\subset \R^n$  define its
projection $k$-volume ratio as

$$
  \mbox{pvr}_k(K) = \sup_L \inf_{P, Q} \vr(PK, QL).
$$
where the supremum is taken over all convex bodies $L\subset \R^n$ and the infimum is taken
over all projections $P, Q$ of rank $k$.
Given two bodies $K$ and $L$, note that the quantity $\inf_{P, Q} \vr(PK, QL)$ measures how close $k$-dimensional projections of the
bodies can be (in terms of the volume ratio). So, $\mbox{pvr}_k(K)$ provides the worst estimate of this measure that works for any body $L$.
In this terminology,  Theorem~\ref{main thm}  says that
$$
 \mbox{pvr}_k(K) \gtrsim
    \, \min\left\{\frac{  k}{  \sqrt{n}} \cdot \sqrt{\frac{1}{\log \log \log(\frac{n\log(n)}{k})}}, \,
  \frac{\sqrt{k}}{\sqrt{\log(\frac{n\log(n)}{k})}}\right\}  ,
$$
while Remark~\ref{rem-l-1} states
$$
      \mbox{pvr} _k (B_1^n) \lesssim
 \begin{cases}
    \frac{k}{\sqrt{n}}  & \text{if } k\geq n^{\nicefrac{2}{3}}  (\log n)^{\nicefrac{2}{3}}\\
    \sqrt{\frac{n}{k}} \log n &\text{if } \sqrt{n} \log n< k\leq n^{\nicefrac{2}{3}}  (\log n)^{\nicefrac{2}{3}}\\
    \sqrt{k} &\text{if }  k \leq  \sqrt{n} \log n.
\end{cases}
$$
Note also that (\ref{GH-vr-orig}) implies that for every convex body $K\subset \R^n$,
$$
  \mbox{pvr}_k (K)\lesssim \sqrt{k}\, \log k,
$$
while (\ref{G-C-P}) implies that
$$
  \mbox{pvr}_k (B_2^n) \geq \inf_{\mbox{rk }Q=k} \vr(B_2^k, Q B_1^n) \gtrsim \sqrt{\frac{k}{\log(2n/k)  }},
$$
which in particular shows that up to logarithmic factors $B_2^n$ maximizes $\mbox{pvr} _k (\cdot)$ and that, in general,
$\mbox{pvr} _k (K)$ could be significantly larger than $k/\sqrt{n}$ even for $k\geq n^{2/3}$.

\smallskip

Finally let us note that it would be natural to consider the following counterpart of the previous parameter. For $1\leq k\leq n$ and a convex body  $K\subset \R^n$ be we define
projection $k$-outer volume ratio as
$$
  \mbox{povr}_k(K) = \sup_L \inf_{P, Q} \vr(PL, QK),
$$
where as before  the supremum is taken over all  convex bodies $L\subset \R^n$ and the infimum is taken
over all projections $P, Q$ of rank $k$.

Note that by
Dvoretzky theorem for $1\leq k \leq c \log n$,
$$
      \mbox{povr}_k(B_2^n) \leq 2.
$$
In the next theorem we show that this quantity is also bounded when $k$ is proportional to $n$.

\begin{theorem}
Let $0 < \lambda \leq 1$  There exists a  constant $C(\lambda)>0$ depending only on $\lambda$ such that
   if $k=\lambda n$ then
   $$
      \mbox{\rm povr}_k(B_2^n) \leq C(\lambda).
   $$
\end{theorem}

\begin{proof}

Let $L\in \R^n$ be a convex body. Applying an affine transformation if needed we can assume
that $L$ is in $M$-position, which means that  $|L|=|B_2^n|$, that $L$ can be covered by $e^{Cn}$
translates of $B_2^n$ and that $B_2^n$ can be covered by $e^{Cn}$ translates of $L$. We refer to  \cite[Chapter~8]{artstein2015asymptotic}  and to  \cite[Chapter~7]{Pisier}
for several equivalent definitions of $M$-position, its existence, and for basic properties
of convex bodies in $M$-positions. Note that the existence of $M$-position in the non-symmetric case was
first established in \cite{MP-m-ell, Rud-m-ell}. By  \cite[Theorem~8.5.4]{artstein2015asymptotic} such a
position exists for every convex body.

Now, Theorem~8.6.1 in  \cite{artstein2015asymptotic} implies that there exists a
projection  $P$ of rank $k=\lambda n$ such that the body $PL$ have the volume ratio bounded by a constant
depending only on $\lambda$ (in fact, it is true for  ``most'' projections).
This implies the desired result.
\end{proof}

\bigskip

\noindent
{\bf Acknowledgement. } Part of this research was done when A.L. was participating in the 
ICERM program {\it ``Harmonic Analysis and Convexity".} A.L. is grateful to ICERM for 
hospitality and excellent working conditions. 

The first-named and third-named authors were partially supported by CONICET-PIP 11220200102366CO
 and ANPCyT PICT 2018-4250.

\end{document}